\documentclass[11pt]{article}

\usepackage{amssymb,amsmath,euscript,bbm,xcolor,graphicx,epstopdf}
\usepackage{verbatim} % to comment out some text
\newtheorem{proposition}{Proposition}[section]

\newtheorem{theorem}[proposition]{Theorem}

\newtheorem{corollary}[proposition]{Corollary}

\def\l{{\langle}}
\def\r{\rangle}

\def\R{{\mathbb R}}
\def\S{{\mathbb S}}

\def\E{{\mathbb E}}
\def\P{{\mathbb P}}

\makeatletter \@addtoreset{equation}{section} \makeatother
\newcommand {\qed}%
{%
    {}\hfill
    {}\hfill
    {$\square $}%
    \vspace {0.3cm}%
    \pagebreak [2]%
    \par
}%
\newenvironment{proof}{%
	\vspace{0.3cm}%
	\pagebreak [2]%
	\par%
	\noindent {\bf  Proof.\  }}{\qed}%
\newenvironment{remark}{%
    \vspace{0.3cm} \pagebreak [2]%
    \par%
    \refstepcounter{proposition}
    \noindent%
    {\bf Remark~\theproposition\  }}{}%
\begin{document}

\title {On critical points of Gaussian random fields under diffeomorphic transformations}
\author{Dan Cheng\thanks{Research partially
		supported by NSF grant DMS-1902432.}\\ Arizona State University
 \and Armin Schwartzman\thanks{Research partially
 	supported by NSF grant DMS-1811659.} \\ University of California, San Diego}

\maketitle

\begin{abstract}
	Let $\{X(t), t\in M\}$ and $\{Z(t'), t'\in M'\}$ be smooth Gaussian random fields parameterized on Riemannian manifolds $M$ and $M'$, respectively, such that $X(t) = Z(f(t))$, where $f: M \to M'$ is a diffeomorphic transformation. We study the expected number and height distribution of the critical points of $X$ in connection with those of $Z$. As an important case, when $X$ is an anisotropic Gaussian random field, then we show that its expected number of critical points becomes proportional to that of an isotropic field $Z$, while the height distribution remains the same as that of $Z$. 
\end{abstract}



\section{Introduction}

The study of critical points of Gaussian random fields, especially of their expected number and height distribution, is important in probability theory \cite{AT07,CL67} and has applications in many areas such as physics \cite{Fyodorov04}, statistics \cite{CS17,CS19sphere}, neuroimaging \cite{Taylor:2007,Pantazis:2005} and astronomy \cite{AstrophysJ85}. However, the exact formulae for the expected number and height distribution of critical points of Gaussian fields are difficult of obtain. The only exception, so far to the authors' knowledge, is the isotropic Gaussian fields \cite{CS15,CS18}. In this paper, we investigate Gaussian fields under diffeomorphic transformations and find a new approach to obtain the exact formulae for a wider class of Gaussian fields.

An important case is the anisotropic Gaussian field, which is a useful model in spatial statistics \cite{Allard:2016}. Anisotropic fields also appear in smoothing of 3D brain images. While isotropic kernels are often used, it has been shown that in some situations anisotropic kernels yield better results \cite{VanHecke:2010}.

The standard approach for solving the expected number and height distribution of critical points of Gaussian random fields is the Kac-Rice formula \cite{AT07,Fyodorov04}. However, the problem becomes much simpler by realizing that an anisotropic field can be obtained by a linear transformation of the index parameter of an isotropic field; see Section \ref{sec:anisotropic} below. We present here an easy solution to the expected number and height distribution of critical points of anisotropic Gaussian fields by casting the problem more generally as a diffeomorphic transformation of the parameter space.

This approach has advantages for other applications where diffeomorphic transformations are used. Diffeomorphic transformations of random fields appear, for example, in structural and functional imaging of the cortical surface of the brain. Because the brain's cortex is a convoluted closed surface, it is often mapped onto the 2-dimensional sphere $\S^2$ \cite{Fischl:1999-II}, where statistical inference can be performed using Gaussian random field theory \cite{Hagler:2006,Pantazis:2005}. Thus, the study of topological features of quantities such as local maxima and local minima of cortical thickness and cortical activity can be carried out in a simpler representation on the sphere.

\section{Gaussian Random Fields under Diffeomorphic Transformations}
Let $\{X(t), t\in M\}$ be a smooth Gaussian random field parameterized on a piecewise $C^2$ compact Riemannian manifold $M$.  Here and in the sequel, the smoothness assumption means that the field satisfies the condition (11.3.1) in \cite{AT07}, which is slightly stronger than $C^2$ but can be implied by $C^3$. The number of critical points of index $i$ of $X$ exceeding level $u$ over the domain $M$ is defined as
\begin{equation}\label{eq:mu-i}
\mu_i^X(M, u) = \# \left\{ t\in M: X(t)\geq u, \nabla X(t)=0, \text{index} (\nabla^2 X(t))=i \right\}, \quad 0\le i\le N,
\end{equation}
where $\nabla X(t)$ and $\nabla^2 X(t)$ are respectively the gradient and Hessian of $X$, and $\text{index} (\nabla^2 X(t))$ denotes the number of negative eigenvalues of $\nabla^2 X(t)$. The expected number of critical points is therefore $\E[\mu_i^X(M, u)]$.

We introduce next the height distribution of a critical point of $X$. It is convenient to consider the tail probability, that is, the probability that the height of the critical point exceeds a fixed threshold at that point, conditioned on the event that the point is a critical point of $X$. Such conditional probability can be defined as
\begin{equation*}
F_i^X(t, u) = \P\{X(t)>u | t \text{ is a critical point of index $i$ of } X \}.
\end{equation*}
It is shown in \cite{CS15} that the height distribution of a critical point of index $i$ of $X$ at $t$ is 
\begin{equation}\label{eq:Palm distr Euclidean}
F_i^X(t, u) = \frac{\E\{|{\rm det} (\nabla^2 X(t))|\mathbbm{1}_{\{X(t)> u\}} \mathbbm{1}_{\{{\rm index}(\nabla^2 X(t))=i\}}|\nabla X(t)=0\}}{\E\{|{\rm det} (\nabla^2 X(t))|\mathbbm{1}_{\{{\rm index}(\nabla^2 X(t))=i\}} | \nabla X(t)=0\}}.
\end{equation}

A mapping $f: M \to M'$ is called $C^2$ diffeomorphic if it is one-to-one and $f$ and its inverse $f^{-1}$ are both twice differentiable. Suppose there exists $\{Z(t'), t'\in M'\}$, a smooth Gaussian random field parameterized on a Riemannian manifold $M'$, and a $C^2$ diffeomorphic transformation $f: M \to M'$, such that $X(t) = Z(f(t))$. Similarly to \eqref{eq:mu-i}, the number of critical points of index $i$ of $Z$ exceeding level $u$ over the domain $M'$ is defined as 
\begin{equation*}
\mu_i^Z(M', u) = \# \left\{ t'\in M': Z(t')\geq u, \nabla Z(t')=0, \text{index} (\nabla^2 Z(t'))=i \right\}, \quad 0\le i\le N;
\end{equation*}
and the height distribution of critical points of index $i$ of $Z$ at $t'$ is defined as $F_i^Z(t', u)$. 

We have the following main result. 

\begin{theorem}\label{thm:main}
	Let $\{X(t), t\in M\}$ and $\{Z(t'), t'\in M'\}$ be smooth Gaussian random fields parameterized on $N$-dimensional, piecewise $C^2$, compact Riemannian manifolds $M$ and $M'$, respectively, such that $X(t) = Z(f(t))$, where $f: M \to M'$ is a $C^2$ diffeomorphic transformation. Then for $0\le i\le N$, $t\in M$ is a critical point of index $i$ of $X$ if and only if $t'=f(t)\in M'$ is a critical point of index $i$ of $Z$; and moreover, for $u\in \R$, 
	\begin{equation}\label{eq:main}
	\E[\mu_i^X(M, u)] = \E[\mu_i^Z(M', u)] \quad {\rm and} \quad F_i^X(t, u) = F_i^Z(t', u).
	\end{equation}
\end{theorem}

\begin{proof}
Denote by $T_t M$ and $T_{t'}M'$ the tangent spaces at $t\in M$ and $t'\in M'$, respectively. By the chain rule,
\begin{equation}\label{eq:chain}
dX|_t=dZ|_{f(t)}df|_t,
\end{equation}
where $dX|_t: T_tM\to \R$, $df|_t: T_tM\to T_{f(t)}M'$ and $dZ|_{f(t)}: T_{f(t)}M'\to \R$. We may write
\begin{equation*}
\begin{split}
dX|_t &= \sum_{i=1}^{N}\frac{\partial X}{\partial x_i} dx_i|_t,\\
dZ|_{f(t)} &= \sum_{i=1}^{N}\frac{\partial Z}{\partial x'_i} dx'_i|_{f(t)},
\end{split}
\end{equation*}
and $df|_t$ is a linear mapping that can be represented as an $N\times N$ matrix, denoted by $B=(B_{ij})_{1\le i,j\le N}$, such that
\begin{equation}\label{eq:df}
df|_t(v)=\sum_{i=1}^{N}\left(\sum_{j=1}^{N}B_{ij}v_j\right)\frac{\partial}{\partial x'_i}|_{f(t)}, \quad \forall v=\sum_{i=1}^{N}v_i\frac{\partial}{\partial x_i}|_t\in T_tM,
\end{equation}
where $\{\frac{\partial}{\partial x_i}|_t\}_{i=1}^n$ and $\{\frac{\partial}{\partial x'_i}|_{f(t)}\}_{i=1}^n$ are the bases in $T_tM$ and $T_{f(t)}M'$, respectively. Let $e_i|_t=\frac{\partial}{\partial x_i}|_t$. By \eqref{eq:chain} and \eqref{eq:df}, we have
\begin{equation*}
\begin{split}
\frac{\partial X}{\partial x_i}|_t&=\l dX|_t, \frac{\partial }{\partial x_i}|_t\r = dX|_t(\frac{\partial }{\partial x_i}|_t) = dX|_t(e_i|_t)\\
&=dZ|_{f(t)}df|_t(e_i|_t) = dZ|_{f(t)}\left(\sum_{k=1}^{N}B_{ki}\frac{\partial }{\partial x'_k}|_{f(t)}\right) = \sum_{k=1}^{N}B_{ki}\frac{\partial Z}{\partial x'_k}|_{f(t)}.
\end{split}
\end{equation*}
Therefore, 
\[
\nabla X|_t = (\frac{\partial X}{\partial x_1}, \ldots, \frac{\partial X}{\partial x_N})^T = B^T(\frac{\partial Z}{\partial x'_1}, \ldots, \frac{\partial Z}{\partial x'_N})^T = B^T\nabla Z|_{f(t)},
\]
implying 
\begin{equation}\label{eq:gradient}
\nabla X|_t=0 \quad \Leftrightarrow \quad \nabla Z|_{f(t)}=0.
\end{equation}
Similarly, we have 
\begin{equation*}
\begin{split}
\frac{\partial }{\partial x_j}(\frac{\partial X}{\partial x_i})|_t = \sum_{l=1}^{N}\sum_{k=1}^{N}B_{ki}\frac{\partial }{\partial x'_l}(\frac{\partial Z}{\partial x'_k})|_{f(t)}B_{lj}.
\end{split}
\end{equation*}
Therefore
\begin{equation}\label{eq:B}
\nabla^2 X|_t = B^T \nabla^2 Z|_{f(t)} B,
\end{equation}
implying 
\begin{equation}\label{eq:Hessian}
{\rm index}(\nabla^2 X|_t) = {\rm index}(\nabla^2 Z|_{f(t)}).
\end{equation}
It follows from \eqref{eq:gradient} and \eqref{eq:Hessian} that $t\in M$ is a critical point of index $i$ of $X$ if and only if $t'=f(t)\in M'$ is a critical point of index $i$ of $Z$.

It follows from \eqref{eq:mu-i}, \eqref{eq:gradient} and \eqref{eq:Hessian} that
\begin{equation*}
\begin{split}
\mu_i^X(M, u) &= \# \left\{ t\in M: X(t)\geq u, \nabla X(t)=0, \text{index} (\nabla^2 X(t))=i \right\}\\
&= \# \left\{ t\in M: Z(f(t))\geq u, \nabla Z|_{f(t)}=0, {\rm index}(\nabla^2 Z|_{f(t)})=i \right\}\\
&= \mu_i^Z(M', u),
\end{split}
\end{equation*}
yielding $\E[\mu_i^X(M, u)] = \E[\mu_i^Z(M', u)]$.

Notice that \eqref{eq:B} implies ${\rm det} (\nabla^2 X(t)) = |{\rm det}(B)|^2{\rm det} (\nabla^2 Z|_{f(t)})$. Therefore, by \eqref{eq:Palm distr Euclidean}, \eqref{eq:gradient} and \eqref{eq:Hessian},
\begin{equation*}
\begin{split}
F_i^X(t, u) &= \frac{\E\{|{\rm det} (\nabla^2 X(t))|\mathbbm{1}_{\{X(t)> u\}} \mathbbm{1}_{\{{\rm index}(\nabla^2 X(t))=i\}}|\nabla X(t)=0\}}{\E\{|{\rm det} (\nabla^2 X(t))|\mathbbm{1}_{\{{\rm index}(\nabla^2 X(t))=i\}} | \nabla X(t)=0\}}\\
&= \frac{|{\rm det}(B)|^2\E\{|{\rm det} (\nabla^2 Z|_{f(t)})|\mathbbm{1}_{\{Z(f(t))> u\}} \mathbbm{1}_{\{{\rm index}(\nabla^2 Z|_{f(t)})=i\}}|\nabla Z|_{f(t)}=0\}}{|{\rm det}(B)|^2\E\{|{\rm det} (\nabla^2 Z|_{f(t)})||\mathbbm{1}_{\{{\rm index}(\nabla^2 Z|_{f(t)})=i\}} | \nabla Z|_{f(t)}=0\}}\\
&= F_i^Z(t', u).
\end{split}
\end{equation*}
\end{proof}

\begin{remark}
	In certain situations, we only know the covariance, or equivalently the distribution, of $X(t)$. But there exist another Gaussian field $Z$ and some diffeomorphic transformation $f$ such that $X(t)$ and $Z(f(t))$ have the same distribution, that is $X(t) \overset{d}{=} Z(f(t))$. Then the results in \eqref{eq:main} still hold since the expected number and height distribution of critical points depend only on the distribution of the field.
\end{remark}

\section{Applications}
\subsection{Anisotropic Gaussian Random Fields on Euclidean Space}\label{sec:anisotropic}

Let $A$ be a nondegenerate $N\times N$ matrix, that is ${\rm det}(A)\neq 0$. An anisotropic Gaussian random field $\{X(t), t\in \R^N\}$ is defined as
\begin{equation}\label{eq:aniso}
X(t)=Z(At), \quad t\in \R^N,
\end{equation}
where $\{Z(t), t\in \R^N\}$ is a smooth, unit-variance, isotropic Gaussian random field. So, it is a special case of the diffeomorphic transformation $f: \R^N \to \R^N$ where $f(t)=At$. Due to the isotropy of $Z$, there exists a function $\rho: [0,\infty) \to \R$ such that the covariance of $Z$ has the form
\begin{equation}\label{eq:isotropy}
\E[Z(t)Z(s)] = \rho\left(\|t-s\|^2\right).
\end{equation}
By \eqref{eq:aniso} and \eqref{eq:isotropy}, equivalently, we may call $\{X(t), t\in \R^N\}$ an anisotropic Gaussian random field if the covariance has the form
\[
\E[X(t)X(s)] = \rho\left(\|A(t-s)\|^2\right).
\]

Using the isotropy property \eqref{eq:isotropy}, it has been shown in Cheng and Schwartzman \cite{CS15,CS18} that the exact formulae of the expected number and height distribution of critical points of the isotropic Gaussian field $Z$ can be obtained by using GOI random matrices. We show below that the study of critical points of anisotropic Gaussian fields can be transferred to isotropic Gaussian fields, so that their expected number and height distribution can be obtained exactly.

\begin{corollary}\label{cor:aniso}
	Let $\{X(t), t\in \R^N\}$ be a smooth, unit-variance, anisotropic Gaussian random field satisfying \eqref{eq:aniso} and let $D$ be an $N$-dimensional Jordan measurable set on $\R^N$. Then for $0\le i\le N$ and $u\in \R$,
	\begin{equation*}
	\E[\mu_i^X(D, u)] = |{\rm det}(A)|\E[\mu_i^Z(D, u)] \quad {\rm and} \quad F_i^X(u) = F_i^Z(u),
	\end{equation*}
	where $\E[\mu_i^Z(D, u)]$ and $F_i^Z(u)$ are respectively the expected number and height distribution for the isotropic Gaussian field $Z$ that can be found in \cite{CS15,CS18}; and $F_i^X(u):=F_i^X(t, u)$ and $F_i^Z(u):=F_i^Z(t, u)$ since they do not depend on $t$ due to the stationarity (isotropy) of $X$ and $Z$.
\end{corollary}
\begin{proof} Let $D' = \{At: t\in D\}$, then ${\rm Vol}(D') = |{\rm det}(A)|{\rm Vol}(D)$. It follows from Theorem \ref{thm:main} and the isotropy of $Z$ that
	\[
	\E[\mu_i^X(D, u)] = \E[\mu_i^Z(D', u)] = |{\rm det}(A)|{\rm Vol}(D)\E[\mu_i^Z([0,1]^N, u)]  = |{\rm det}(A)|\E[\mu_i^Z(D, u)]
	\]
	and
	\[
	F_i^X(u) = F_i^X(t, u) = F_i^Z(At, u) = F_i^Z(t, u) = F_i^Z(u).
	\]
\end{proof}	

\begin{remark}
	Corollary \ref{cor:aniso} shows that the expected number of critical points of the anisotropic field $X$ is that of the isotropic field $Z$ scaled by $|{\rm det}(A)|$, while the height distribution remains the same as that of $Z$. A special case when $A$ is a diagonal matrix was considered in \cite{Fyodorov15} with tedious computation.
	
	It can also be seen easily that the results in Corollary \ref{cor:aniso} hold in a more general setting where $Z$ is stationary but not necessarily isotropic.
\end{remark}

\subsection{Gaussian Random Fields on Ellipsoids}

Suppose $\{X(t), t\in M\}$ is a smooth Gaussian field on an $N$-dimensional ellipsoid $M$. If $M$ is embedded in $\R^{N+1}$, then there is a linear mapping on $\R^{N+1}$, denoted by $g$, that maps $M$ onto the $N$-dimensional unit sphere $\S^N$. Let $f: M \to \S^N$ be the $N$-dimensional diffeomorphic mapping between $M$ and $\S^N$ induced by $g$. Then we can write $X(t) = Z(f(t))$, where $\{Z(t'), t'\in \S^N\}$ is a Gaussian random field on the $N$-dimensional unit sphere $\S^N$. Now the study of critical points of $X$ on the ellipsoid $M$ can be transferred to the Gaussian field $Z$ on the unit sphere $\S^N$ by Theorem \ref{thm:main}.

In particular, if the Gaussian field $Z$ above is isotropic on $\S^N$, then by Theorem \ref{thm:main}, $\E[\mu_i^X(M, u)] = \E[\mu_i^Z(\S^N, u)]$ and $F_i^X(t, u) = F_i^Z(t', u)$ with $t'=f(t)$, whose formulae can be found in \cite{CS15,CS18}.

\section*{Acknowledgment}

The first author thanks Professor Yan Fyodorov for helpful discussions.

\begin{small}

\end{small}

\bigskip

\begin{quote}
\begin{small}

\textsc{Dan Cheng}\\
School of Mathematical and Statistical Sciences \\
Arizona State University\\
900 S. Palm Walk\\
Tempe, AZ 85281, U.S.A.\\
E-mail: \texttt{cheng.stats@gmail.com}

\vspace{.1in}
		
\textsc{Armin Schwartzman}\\
Division of Biostatistics and Bioinformatics and \\
Halicio\v{g}lu Data Science Institute \\
University of California San Diego\\
9500 Gilman Dr.\\
La Jolla, CA 92093, U.S.A.\\
E-mail: \texttt{armins@ucsd.edu}

	\end{small}
\end{quote}

\end{document}